\documentclass[oneside,english]{amsart}
\usepackage{libertineRoman}
\usepackage{avant}
\usepackage{libertineMono}
\usepackage[T1]{fontenc}
\usepackage[latin9]{inputenc}
\usepackage{babel}
\usepackage{pifont}
\usepackage{mathtools}
\usepackage{enumitem}
\usepackage{amstext}
\usepackage{amsthm}
\usepackage{amssymb}
\usepackage{setspace}
\usepackage[unicode=true,pdfusetitle,
 bookmarks=true,bookmarksnumbered=false,bookmarksopen=false,
 breaklinks=false,pdfborder={0 0 0},pdfborderstyle={},backref=false,colorlinks=false]
 {hyperref}

\makeatletter
\theoremstyle{plain}
\newtheorem{thm}{\protect\theoremname}
\theoremstyle{remark}
\newtheorem{rem}[thm]{\protect\remarkname}
\theoremstyle{plain}
\newtheorem{prop}[thm]{\protect\propositionname}
\theoremstyle{plain}
\newtheorem{lem}[thm]{\protect\lemmaname}
\theoremstyle{plain}
\newtheorem{fact}[thm]{\protect\factname}
\theoremstyle{plain}
\newtheorem{cor}[thm]{\protect\corollaryname}
\theoremstyle{plain}
\newtheorem{question}[thm]{\protect\questionname}

\AtBeginDocument{

}

\makeatother

\providecommand{\corollaryname}{Corollary}
\providecommand{\factname}{Fact}
\providecommand{\lemmaname}{Lemma}
\providecommand{\propositionname}{Proposition}
\providecommand{\questionname}{Question}
\providecommand{\remarkname}{Remark}
\providecommand{\theoremname}{Theorem}

\begin{document}
\global\long\def\fii{\varphi}%
\global\long\def\T{\mathbb{T}}%
\global\long\def\F{\mathbb{F}}%
\global\long\def\M{\mathcal{M}}%
\global\long\def\zfc{\mathrm{ZFC}}%
\global\long\def\N{\mathcal{N}}%
\global\long\def\p{\mathbb{P}}%
\global\long\def\q{\mathbb{Q}}%
\global\long\def\NN{\mathbb{N}}%
\global\long\def\RR{\mathbb{R}}%
\global\long\def\power{\mathcal{P}}%
\global\long\def\LL{\mathcal{L}}%
\global\long\def\K{\mathcal{K}}%
\global\long\def\A{\mathcal{A}}%

\global\long\def\hod{\mathrm{HOD}}%
\global\long\def\Def{\operatorname{Def}}%
\global\long\def\zfc{\mathrm{ZFC}}%
\global\long\def\im{\operatorname{Im}}%
\global\long\def\dom{\operatorname{Dom}}%
\global\long\def\rng{\operatorname{Range}}%
\global\long\def\ord{{\bf Ord}}%
\global\long\def\lp{\operatorname{lp}}%
\global\long\def\meas{\operatorname{meas}}%
\global\long\def\card{{\bf Card}}%
\global\long\def\lex{\mathrm{Lex}}%
\global\long\def\id{\operatorname{Id}}%
\global\long\def\ult{\operatorname{Ult}}%
\global\long\def\It{\operatorname{It}}%

\global\long\def\th#1{\mathrm{Th}\left(#1\right)}%
\global\long\def\crit#1{\mathrm{crit}\left(#1\right)}%
\global\long\def\cof#1{\mathrm{cf}\left(#1\right)}%
\global\long\def\fin{\power_{\mathrm{Fin}}}%
\global\long\def\cfq{Q^{\mathrm{cf}}}%
\global\long\def\otp{\operatorname{otp}}%
\global\long\def\fin#1{\left[#1\right]^{<\omega}}%
\global\long\def\psu#1#2{\prescript{#1}{}{#2}}%

\global\long\def\con{\subseteq}%
\global\long\def\til{,\dots,}%
\global\long\def\emp{\varnothing}%
\global\long\def\smin{\mathord{\smallsetminus}}%
\global\long\def\aa{\mathtt{aa}}%
\global\long\def\mets{\mathord{\upharpoonright}}%

\global\long\def\sdiff{\triangle}%
\global\long\def\po{\trianglelefteq}%
\global\long\def\spo{\vartriangleleft}%
\global\long\def\pdwn{\mathord{\downarrow}}%
\global\long\def\pup{\mathord{\uparrow}}%
\global\long\def\nec{\square}%
\global\long\def\pos{\lozenge}%
\global\long\def\necc{\square_{ccc}}%
\global\long\def\posc{\lozenge_{ccc}}%

\title{Iterating the cofinality-$\omega$ constructible model}
\author{Ur Ya'ar}
\address{\parbox{\linewidth}{Einstein Institute of Mathematics \\ University of Jerusalem \\	Edmond J. Safra Campus, Givat Ram \\ Jerusalem 91904, ISRAEL\\}}
\email{ur.yaar@mail.huji.ac.il}
\thanks{I would like to thank my advisor, Prof. Menachem Magidor, for his
guidance and support without which this work would not have been possible.}
\begin{abstract}
We investigate iterating the construction of $C^{*}$, the $L$-like
inner model constructed using first order logic augmented with the
``cofinality $\omega$'' quantifier. We first show that $\left(C^{*}\right)^{C^{*}}=C^{*}\ne L$
is equiconsistent with $\zfc$, as well as having finite strictly
decreasing sequences of iterated $C^{*}$s. We then show that in models
of the form $L^{\mu}$ we get infinite decreasing sequences of length
$\omega$, and that an inner model with a measurable cardinal is required
for that.
\end{abstract}

\maketitle

\section{Introduction}

The model $C^{*}$, introduced by Kennedy, Magidor and V\"{a}\"{a}n\"{a}nen
 in \cite{IMEL}, is the model of sets constructible using the logic
$\LL(Q_{\omega}^{\mathrm{cf}})$ -- first order logic augmented with
the ``cofinality $\omega$'' quantifier. As in the case of $L$
-- the model of sets constructible using first order logic, this
is a model of $\zfc$, and one can phrase the formula ``$V=C^{*}$'',
i.e. $\forall x\exists\alpha(x\in L'_{\alpha})$ where $L'_{\alpha}$
is the $\alpha$-th level in the construction of $C^{*}$. Unlike
$L$, however, it is not always true that $C^{*}\vDash V=C^{*}$,
which is equivalent to the question whether $\left(C^{*}\right)^{C^{*}}=C^{*}$.
This is clearly the case if $V=L$, so the interesting question is
whether this can hold with $C^{*}\ne L$. In section \ref{sec:V=00003DC*}
we show that this is consistent relative to the consistency of $\zfc$.
Next we investigate the possibilities of $C^{*}\nvDash V=C^{*}$.
In such a case, it makes sense to define recursively the iterated
$C^{*}$s: 
\begin{align*}
C^{*0} & =V\\
C^{*\left(\alpha+1\right)} & =\left(C^{*}\right)^{C^{*\alpha}}\,\text{ for any \ensuremath{\alpha}}\\
C^{*\alpha} & =\bigcap_{\beta<\alpha}C^{*\beta}\,\,\text{ for limit \ensuremath{\alpha}}.
\end{align*}
 This type of construction was first investigated by McAloon \cite{mcaloon2}
regarding $\mathrm{HOD}$, where he showed that it is equiconsistent
with $\zfc$ that there is a strictly decreasing sequence of iterated
$\mathrm{HOD}$ of length $\omega$, and the intersection of the sequence
can be either a model of $\zfc$ or of $\mathrm{ZF+\neg AC}$. Harrington
also showed (in unpublished notes, cf. \cite{ZADROZNY}) that the
intersection might not even be a model of $\mathrm{ZF}$. Jech \cite{jech1975descending}
showed that it is possible to have a strictly decreasing sequence
of iterated $\mathrm{HOD}$ of any arbitrary ordinal length, and later
Zadro\.{z}ny \cite{zadrozny1981transfinite} improved this to get
an $\ord$ length sequence. In section \ref{sec:Iterating C*} we
show that unlike the case of $\mathrm{HOD}$, without large cardinals
we can only have finite decreasing sequences of iterated $C^{*}$,
and that assuming the existence of a measurable cardinal is equivalent
to the consistency of a strictly decreasing sequence of length $\omega$.

\section{\label{sec:V=00003DC*}Relative consistency of \textquotedblleft$V=C^{*}\protect\ne L$\textquotedblright}

In this section we follow the method of Zadro\.{z}ny \cite{ZADROZNY}
to obtain the following result:
\begin{thm}
\label{thm:V=00003DC*}If $\zfc$ is consistent then so is $\zfc+V=C^{*}\ne L\,+$
$2^{\aleph_{0}}=\aleph_{1}$.
\end{thm}

The idea\footnote{I'd like to thank Kameryn Williams for his exposition of this and
related results in his blog -- \href{http://kamerynjw.net/2019/12/04/omegath-hod.html}{http://kamerynjw.net/2019/12/04/omegath-hod.html}} of Zadro\.{z}ny's proofs, which are based on results of McAloon's
\cite{McAloon1,mcaloon2}, is to add a generic object (to make $V\ne L$),
code it using some other generic object, then code the coding, and
so on, iterating until we catch our tail. Our coding tool will be
the modified Namba forcing of \cite[section 6]{IMEL}, which adds
a countable cofinal sequence to any element of some countable sequence
of regular cardinals $>\aleph_{1}$ (and only to them). Revised countable
support iterations of this forcing preserves $\omega_{1}$. In \cite[theroem 6.7]{IMEL},
these tools are used to produce a model of $\zfc+V=C^{*}\ne L\,+$
$2^{\aleph_{0}}=\aleph_{2}$, but this requires an inaccessible cardinal
(as proven there as well).

These two results covers all possibilities, since in \cite[corollary to theroem 5.20]{IMEL},
it is shown that the statement $V=C^{*}$ implies that $2^{\aleph_{0}}\in\left\{ \aleph_{1},\aleph_{2}\right\} $
and for any $\kappa>\aleph_{0}$ $2^{\kappa}=\kappa^{+}$ .

To prove theorem \ref{thm:V=00003DC*}, we begin with $V=L\left[A_{0}\right]$
where $A_{0}$ is a countable set of ordinals such that $A_{0}$ does
not contain any of its limit points. Set $\p_{0}=\left\{ 1\right\} $.
Inductively we assume that $\p_{n}$ forces the existence of a countable
set of ordinals $A_{n}$, $\xi_{n}=\sup A_{n}$ and we set $\p_{n+1}=\p_{n}*\dot{\q}_{n+1}$
where $\dot{\q}_{n+1}$ is the modified Namba forcing to add a Namba
sequence $E_{\alpha}$ to each $\aleph_{\xi_{n}+\alpha+2}^{L}$ such
that $\alpha\in A_{n}$. We can require that $E_{\alpha}\con\left(\aleph_{\xi_{n}+\alpha+1}^{L},\aleph_{\xi_{n}+\alpha+2}^{L}\right)$,
so that $A_{n+1}:=\bigcup\left\langle E_{\alpha}\mid\alpha\in A_{n}\right\rangle $
does not contain any of its limit points. $\p_{\omega}$ is the full
support (which is in our case also the revised countable support)
iteration. Let $G\con\p_{\omega}$ generic, and denote $A=\bigcup_{n<\omega}A_{n}$.
By the properties of the modified Namba forcing, for any $\gamma$,
$V\left[G\right]\vDash\cof{\aleph_{\gamma+2}^{L}}=\omega$ iff $\gamma=\xi_{n}+\alpha$
for $\alpha\in A_{n}$. 
\begin{rem}
\label{rem:codings}
\begin{enumerate}
\item \label{rem:enu: suprema of codings}$\xi_{n+1}=\aleph_{\xi_{n}\cdot2}^{L}$
, so inductively depends only $A_{0}$ and not on the generics. $\xi_{\omega}:=\sup A$
satisfies $\xi_{\omega}=\aleph_{\xi_{\omega}}^{L}$ and is of cofinality
$\omega$.
\item \label{rem:enu:reconstruct An}From $\otp(A_{0})$ and $A$ we can
inductively reconstruct each $A_{n}$ -- $A_{0}$ is the first $\otp(A_{0})$
elements of $A$, and if we know $A_{n}$, then $A_{n+1}$ are the
first $\otp(A_{n})\cdot\omega$ elements of $A$ above $\sup A_{n}$.
\item \label{rem:enu:otp}Hence for each $n$, $\otp(A_{n})=\otp(A_{0})\cdot\omega^{n}$.
\item \label{rem:enu:reconstruct E_alpha}If $\alpha\in A_{n}$, then $E_{\alpha}=A\cap\left(\aleph_{\xi_{n}+\alpha+1}^{L},\aleph_{\xi_{n}+\alpha+2}^{L}\right)$.
\end{enumerate}
\end{rem}

\begin{prop}
$\left(C^{*}\right)^{V\left[G\right]}=L\left[A\right]=V\left[G\right]$
\begin{proof}
By the properties of the modified Namba forcing, at each stage of
the iteration the only cardinals of $L\left[A_{0}\right]$ receiving
cofinality $\omega$ are the ones in $A_{n}$. The whole iteration
will also add new $\omega$ sequences to $\sup A$, but this already
had cofinality $\omega$ as we noted earlier. So $V[G]\vDash\cof{\gamma}=\omega$
iff either $V\vDash\cof{\gamma}=\omega$ or there is $n$ s.t. $V\vDash\cof{\gamma}=\aleph_{\xi_{n}+\alpha+2}^{L}$
for $\alpha\in A_{n}$. And on the other hand, if $\alpha\in A$ then
$V\left[G\right]\vDash\cof{\aleph_{\xi_{n}+\alpha+2}^{L}}=\omega$\,.
So 
\[
A=\bigcup_{n<\omega}\left\{ \alpha\in[\xi_{n-1},\xi_{n})\mid V\left[G\right]\vDash\cof{\aleph_{\xi_{n}+\alpha+2}^{L}}=\omega\right\} \in\left(C^{*}\right)^{V\left[G\right]}
\]
(where $\xi_{-1}=0$\,) hence $L\left[A\right]\con\left(C^{*}\right)^{V\left[G\right]}$.

As we noted, for every $\alpha\in A$, $E_{\alpha}$ can be reconstructed
from $A$ and $\alpha$, so $\left\langle E_{\alpha}\mid\alpha\in A\right\rangle $
is in $L\left[A\right]$. $G$ can be reconstructed from this sequence,
hence $V\left[G\right]\con L\left[A\right]\con\left(C^{*}\right)^{V\left[G\right]}$,
so the equality follows.
\end{proof}
\end{prop}

This finishes the proof of theorem \ref{thm:V=00003DC*} since for
any non-empty $A_{0}$ we'll get a model of ``$V=C^{*}\ne L$'',
and $2^{\aleph_{0}}=\aleph_{1}$ still holds.

Before moving to the next section we prove a useful lemma:
\begin{lem}
\label{lem:C*=00003DL=00005BE=00005D}Let $E=\left\{ \alpha<\omega_{2}^{V}\mid\cof{\alpha}=\omega\right\} $.
\begin{enumerate}
\item \label{enu:no 0=000023}If $0^{\sharp}$ does not exist, then $C^{*}=L\left[E\right]$.
\item \label{enu:no L-mu}If there is no inner model with a measurable cardinal,
then $C^{*}=K\left[E\right]$ where $K$ is the Dodd-Jensen core model.
\end{enumerate}
\end{lem}

\begin{proof}
1. Clearly $E\in C^{*}$ so $L\left[E\right]\con C^{*}$. Let $\alpha\in\ord$.
If $\cof{\alpha}\geq\omega_{2}^{V}$ then also $L\vDash\cof{\alpha}\geq\omega_{2}^{V}$
so in particular $L\vDash\cof{\alpha}>\omega$. If $\cof{\alpha}<\omega_{2}^{V}$,
let $A\con\alpha$ be cofinal, so $\left|A\right|\leq\aleph_{1}$.
By the covering theorem, there is $B\in L$, $A\con B\con\alpha$
s.t. $\left|B\right|=\aleph_{1}+\left|A\right|=\aleph_{1}$. Let $\bar{\alpha}=\otp(B)$,
so $\bar{\alpha}<\omega_{2}^{V}$ , and $\cof{\alpha}=\cof{\bar{\alpha}}$,
so $\cof{\alpha}=\omega$ iff $\bar{\alpha}\in E$. 

To summarize, we get that for every $\alpha$, in $L\left[E\right]$
we can determine whether $\cof{\alpha}=\omega$ or not, so $C^{*}\con L\left[E\right]$.

2. The proof is exactly the same, noting that $K\con C^{*}$ by \cite[theroem 5.5]{IMEL},
and that our assumption implies the covering theorem holds for $K$.
\end{proof}

\section{\label{sec:Iterating C*}Iterating $C^{*}$ }
\begin{thm}
If $\zfc$ is consistent then so is the existence of a model with
a decreasing $C^{*}$-sequence of any finite length.
\end{thm}

\begin{proof}
Going back to the proof of theorem \ref{thm:V=00003DC*}, we note
that for any $n$, $\left(C^{*}\right)^{L\left[\bigcup_{k=0}^{n+1}A_{k}\right]}=L\left[\bigcup_{k=0}^{n}A_{k}\right]$:
$A_{n}$ can be computed from $A_{n+1}$ using the cofinality-$\omega$
quantifier, which gives $\supseteq$, and on the other hand, from
$\bigcup_{k=0}^{n}A_{k}$ we know exactly which ordinals will have
cofinality $\omega$ in $L\left[\bigcup_{k=0}^{n+1}A_{k}\right]$,
which gives $\subseteq$. So by starting e.g. from $A_{0}=\omega$,
$L\left[\bigcup_{k=0}^{n}A_{k}\right]$ has the decreasing $C^{*}$
chain 
\[
L\left[\bigcup_{k=0}^{n}A_{k}\right]=C^{*0}\supsetneq C^{*1}\supsetneq...\supsetneq C^{*n}=L.\qedhere
\]
\end{proof}
Without large cardinals this is best possible:
\begin{thm}
If there is no inner model with a measurable cardinal, then there
is $k<\omega$ such that $C^{*k}=C^{*(k+1)}$.
\end{thm}

\begin{proof}
By applying lemma \ref{lem:C*=00003DL=00005BE=00005D}.\ref{enu:no L-mu}
inside each $C^{*n}$, for every $n$ we have $C^{*(n+1)}=K\left[E_{n}\right]$
where 
\[
E_{n}=\left\{ \alpha<\omega_{2}^{C^{*n}}\mid C^{*n}\vDash\cof{\alpha}=\omega\right\} .
\]
The sequence $\left\langle \left(\omega_{1}^{C^{*n}},\omega_{2}^{C^{*n}}\right)\mid n<\omega\right\rangle $
is non-increasing in both coordinates, hence it stabilizes. Let $k$
such that $(\omega_{1}^{C^{*k}},\omega_{2}^{C^{*k}})=(\omega_{1}^{C^{*(k+1)}},\omega_{2}^{C^{*(k+1)}})$,
and we claim that $C^{*(k+1)}=C^{*(k+2)}$. To simplify notation we
assume w.l.o.g $k=0$, i.e $\omega_{i}^{C^{*}}=\omega_{i}^{V}$ for
$i=1,2$ (so we can omit the superscript) and we want to show that
$\left(C^{*}\right)^{C*}=C^{*}$. We have:
\begin{align*}
C^{*} & =K\left[\left\{ \alpha<\omega_{2}\mid V\vDash\cof{\alpha}=\omega\right\} \right]=K\left[E_{0}\right]\\
\left(C^{*}\right)^{C*} & =K\left[\left\{ \alpha<\omega_{2}\mid C^{*}\vDash\cof{\alpha}=\omega\right\} \right]=K\left[E_{1}\right].
\end{align*}
 Clearly $E_{1}\con E_{0}$. On the other hand, if $\alpha\in\omega_{2}\smin E_{1}$,
this means that $C^{*}\vDash\cof{\alpha}=\omega_{1}$, and since $\omega_{1}^{C^{*}}=\omega_{1}$,
we get that also $V\vDash\cof{\alpha}=\omega_{1}$, so $\alpha\in\omega_{2}\smin E_{0}$,
thus $E_{1}=E_{0}$, and our claim is proved.
\end{proof}
Our next goal is to show that this is precisely the consistency strength
of a decreasing sequence:
\begin{thm}
\label{thm:decreasing}If there is an inner model with a measurable
cardinal, then it is consistent that the sequence $\left\langle C^{*n}\mid n<\omega\right\rangle $
is strictly decreasing. 
\end{thm}

We work in $L^{\mu}$ where $\mu$ is a measure on $\kappa$, and
denote by $M_{\alpha}$ the $\alpha$-th iterate of $L^{\mu}$ by
$\mu$, $j_{\alpha,\beta}:M_{\alpha}\to M_{\beta}$ the elementary
embedding and $\kappa_{\alpha}=j_{0,\alpha}(\kappa)$.

In \cite[theroem 5.16]{IMEL} the authors show that if $V=L^{\mu}$
then $C^{*}=M_{\omega^{2}}\left[E\right]$ where $E=\left\{ \kappa_{\omega\cdot n}\mid n<\omega\right\} $.
We improve this by showing that $C^{*}$ is unchanged after adding
a Prikry sequence to $\kappa$, and then investigate the $C^{*}$-chain
of $L^{\mu}$. First we prove two useful lemmas.
\begin{lem}
\label{lem:Prikry-generic}$E=\left\{ \kappa_{\omega\cdot n}\mid n<\omega\right\} $
is generic over $M_{\omega^{2}}$ for the Prikry forcing on $\kappa_{\omega^{2}}$
defined from the ultrafilter $U^{(\omega^{2})}\in M_{\omega^{2}}$.
\end{lem}

\begin{proof}
We use Mathias's characterization of Prikry forcing:
\begin{fact}[Mathias, cf. \cite{mathias1973prikry}]
 Let $M$ be a transitive model of $\zfc$, $U$ a normal ultrafilter
on $\kappa$, then $S\con\kappa$ of order type $\omega$ is generic
over $M$ for the Prikry forcing defined from $U$ iff for any $X\in U$,
$S\smin X$ is finite.
\end{fact}

So we need to show that for any $X\in U^{(\omega^{2})}$, $E\smin X$
is finite. The ultrafilter $U^{(\omega^{2})}$ is defined by $X\in U^{(\omega^{2})}$
iff $\exists\alpha<\kappa_{\omega^{2}}$$\left\{ \kappa_{\beta}\mid\alpha\leq\beta<\kappa_{\omega^{2}}\right\} \con X$.
For $X\in U^{(\omega^{2})}$, choose some $\alpha<\kappa_{\omega^{2}}$
such that $\left\{ \kappa_{\beta}\mid\alpha\leq\beta<\kappa_{\omega^{2}}\right\} \con X$,
then $E\smin\alpha=\left\{ \kappa_{\omega\cdot n}\mid\omega\cdot n<\alpha\right\} $,
and since $\kappa_{\omega^{2}}=\sup\left\{ \kappa_{\omega\cdot n}\mid n<\omega\right\} $,
this set is finite. Hence $E=\left\{ \kappa_{\omega\cdot n}\mid n<\omega\right\} $
satisfies the characterization.
\end{proof}
\begin{lem}
\label{lem:k-cof preserves}For any $\beta<\kappa$ and any $\alpha$,
$M_{\beta}\vDash\cof{\alpha}=\kappa$ iff $V\vDash\cof{\alpha}=\kappa$.
\end{lem}

\begin{proof}
 $\kappa$ is regular in $V$, thus it is regular in every $M_{\beta}$
which is an inner model of $V$. If $M_{\beta}\vDash\cof{\alpha}=\kappa$,
then there is a cofinal $\kappa$-sequence in $\alpha$ (in both $M_{\beta}$
and $V$), and since $\kappa$ is regular we get $V\vDash\cof{\alpha}=\kappa$.
If $V\vDash\cof{\alpha}=\kappa$, then the same argument rules out
$M_{\beta}\vDash\cof{\alpha}<\kappa$. So the only case left to rule
out is $V\vDash\cof{\alpha}=\kappa$ $\land$ $M_{\beta}\vDash\cof{\alpha}>\kappa$.
If $\beta=\gamma+1$, then $M_{\beta}$ is contained in $M_{\gamma}$
and closed under $\kappa$-sequences in it, so they agree on cofinality
$\kappa$, and by induction we get that they agree with $V$ as well.
So assume $\beta$ is limit and let $\left\langle \alpha_{\eta}\mid\eta<\kappa\right\rangle $
be a cofinal sequence in $\alpha$. by definition of the limit ultrapower,
each $\alpha_{\eta}$ is of the form $j_{\bar{\beta},\beta}(\bar{\alpha}_{\eta})$
for some $\bar{\beta}<\beta$. We can also assume that each such $\bar{\beta}$
is large enough so that $\alpha\in\rng(j_{\bar{\beta},\beta})$. Since
$\beta<\kappa$, there is some $\bar{\beta}$ fitting $\kappa$ many
$\alpha_{\eta}$s, so w.l.o.g we can assume $\bar{\beta}$ fits all
of them. We can assume $\bar{\beta}>0$ so $\kappa$ is a fixed point
of $j_{\bar{\beta},\beta}$. If $\bar{\alpha}=\sup\left\{ \bar{\alpha}_{\eta}\mid\eta<\kappa\right\} $\footnote{Note that we can't assume this sequence is in $M_{\bar{\beta}}$}
then, since $\alpha=\sup\left\{ j_{\bar{\beta},\beta}(\bar{\alpha}_{\eta})\mid\eta<\kappa\right\} $
and $\alpha\in\rng(j_{\bar{\beta},\beta})$, we must have that $\alpha=j_{\bar{\beta},\beta}(\bar{\alpha})$.
$\bar{\alpha}$ is of cofinality $\kappa$ in $V$, so by induction
also in $M_{\bar{\beta}}$, hence by elementarity $M_{\beta}\vDash\cof{\alpha}=\kappa$.
\end{proof}
\begin{prop}
\label{prop:C* in Prikry ext}If $V=L^{\mu}$ where $\mu$ is a measure
on $\kappa$, $G$ is generic for Prikry forcing on $\kappa$, then
$\left(C^{*}\right)^{V\left[G\right]}=\left(C^{*}\right)^{V}$.
\end{prop}

\begin{proof}
After forcing with Prikry forcing, the only change of cofinalities
is that $\kappa$ becomes of cofinality $\omega$. So $V\left[G\right]\vDash\cof{\alpha}=\omega$
iff $V\vDash\cof{\alpha}\in\left\{ \omega,\kappa\right\} $. We now
follow the proof of \cite[theroem 5.16]{IMEL}. 

Consider $M_{\omega^{2}}$, the $\omega^{2}$ iterate of $V$, and
let $E=\left\{ \kappa_{\omega\cdot n}\mid n<\omega\right\} $ and
fix an ordinal $\alpha$. As in the proof of \cite[theroem 5.16]{IMEL},
$V\vDash\cof{\alpha}=\omega$ iff ($M_{\omega^{2}}\left[E\right]\vDash\cof{\alpha}\in\left\{ \omega\right\} \cup E\cup\left\{ \sup E\right\} $).
Regarding cofinality $\kappa$ -- by lemma \ref{lem:k-cof preserves},
$V\vDash\cof{\alpha}=\kappa$ iff $M_{\omega^{2}}\vDash\cof{\alpha}=\kappa$.
As we noted, $M_{\omega^{2}}=L^{\nu}$ where $\nu$ is a measure on
$\kappa_{\omega^{2}}$ and by lemma \ref{lem:Prikry-generic}, $E$
is Prikry generic over it, hence, since cofinality $\kappa$ is unaffected
by Prikry forcing on $\kappa_{\omega^{2}}$, we get $V\vDash\cof{\alpha}=\kappa$
iff $M_{\omega^{2}}\left[E\right]\vDash\cof{\alpha}=\kappa$. Putting
these facts together, in $M_{\omega^{2}}\left[E\right]$ we can detect
whether $V\vDash\cof{\alpha}\in\left\{ \omega,\kappa\right\} $, so
we know whether $V\left[G\right]\vDash\cof{\alpha}=\omega$, hence
we can construct $\left(C^{*}\right)^{V\left[G\right]}$ inside $M_{\omega^{2}}\left[E\right]$. 

The other direction of the proof is almost the same as in \cite[theroem 5.16]{IMEL}:
$E$ is the set of ordinals in the interval $(\kappa,\kappa_{\omega^{2}})$
which have cofinality $\omega$ in $V\left[G\right]$ and are regular
in the core model\footnote{Note that here we had to avoid $\kappa$ which also satisfies this.},
which is contained in any $C^{*}$, so $E\in\left(C^{*}\right)^{V\left[G\right]}$,
and from $E$ one can define $M_{\omega^{2}}$, so $M_{\omega^{2}}[E]\con\left(C^{*}\right)^{V\left[G\right]}$.
\end{proof}
Now we can analyze the $C^{*}$-chain of $V=L^{\mu}$. To avoid confusion
we stick to the notation $C^{*\alpha}$, starting from $C^{*0}=L^{\mu}=M_{0}$,
with $M_{\alpha}$ being the $\alpha$th iterate of $L^{\mu}$ and
$\kappa_{\alpha}$ the $\alpha$th image of the measurable cardinal.
So by \cite[theroem 5.16]{IMEL} we have $C^{*1}=M_{\omega^{2}}\left[E_{1}\right]$
for $E_{1}=\left\{ \kappa_{\omega\cdot n}\mid n<\omega\right\} $.
As we noted earlier $M_{\omega^{2}}$ is also of the form $L^{\mu'}$
for the measurable $\kappa_{\omega^{2}}$, and by lemma \ref{lem:Prikry-generic},
$E_{1}$ is Prikry generic over it, so by proposition \ref{prop:C* in Prikry ext}
\[
C^{*2}=\left(C^{*}\right)^{M_{\omega^{2}}[E]}=\left(C^{*}\right)^{M_{\omega^{2}}}
\]
which is again the $\omega^{2}$-th iterate of $M_{\omega^{2}}$,
i.e. $M_{\omega^{2}+\omega^{2}}$, plus the corresponding sequence
-- $E_{2}=\left\{ \kappa_{\omega^{2}+\omega\cdot n}\mid n<\omega\right\} $.
So $C^{*2}=M_{\omega^{2}\cdot2}\left[E_{2}\right]$. We can continue
inductively, and get the following:
\begin{thm}
If $V=L^{\mu}$ then for every $m\geq1$ $C^{*m}=M_{\omega^{2}\cdot m}\left[E_{m}\right]$
where $M_{\alpha}$ is the $\alpha$th iterate of $V$, $\kappa_{\alpha}$
the $\alpha$th image of the measurable cardinal and $E_{m}=\left\{ \kappa_{\omega^{2}\cdot(m-1)+\omega\cdot n}\mid n<\omega\right\} $. 
\end{thm}

This concludes the proof of theorem \ref{thm:decreasing}. To analyze
$\left(C^{*\omega}\right)^{L^{\mu}}$, we will use the following theorem,
due to Bukovsk\'{y} \cite{bukovsky1973changing,bukovsky1977iterated}
and Dehornoy \cite{dehornoy1978iterated}:
\begin{fact}
\label{thm:Buk-Deho}If $\kappa$ is measurable, $M_{\alpha}$ is
the $\alpha$-th iterate of $V$as by a normal ultrafilter on $\kappa$
and $\kappa_{\alpha}$ the $\alpha$-th image of $\kappa$, then for
any limit ordinal $\lambda$ exactly one of the following holds:
\begin{enumerate}
\item If $\exists\alpha<\lambda$ s.t. $M_{\alpha}\vDash\cof{\lambda}>\omega$
then $\bigcap_{\alpha<\lambda}M_{\alpha}=M_{\lambda}$
\item If $\lambda=\alpha+\omega$ for some $\alpha$, then $\left\langle \kappa_{\alpha+n}\mid n<\omega\right\rangle $
is Prikry generic over $M_{\lambda}$ and $\bigcap_{\alpha<\lambda}M_{\alpha}=M_{\lambda}\left[\left\langle \kappa_{\alpha+n}\mid n<\omega\right\rangle \right]$
\item Otherwise, $\bigcap_{\alpha<\lambda}M_{\alpha}$ is a quasi-generic
extension of $M_{\lambda}$, hence satisfies $\mathrm{ZF}$, but doesn't
satisfy $\mathrm{AC}$.
\end{enumerate}
\end{fact}

\begin{cor}
$C^{*\omega}=\bigcap_{\alpha<\omega^{3}}M_{\alpha}$ and it satisfies
$\mathrm{ZF}$ but not $\mathrm{AC}$.
\end{cor}

\begin{proof}
By definition and our previous calculation, 
\[
C^{*\omega}=\bigcap_{m<\omega}C^{*m}=\bigcap_{1\leq m<\omega}M_{\omega^{2}\cdot m}\left[E_{m}\right]
\]
and for each $m\geq1$, $E_{m}\in M_{\omega^{2}\cdot(m-1)}$ so 
\[
\bigcap_{1\leq m<\omega}M_{\omega^{2}\cdot m}\left[E_{m}\right]=\bigcap_{m<\omega}M_{\omega^{2}\cdot m}=\bigcap_{\alpha<\omega^{3}}M_{\alpha}.
\]
 Since $\omega^{3}$ is of cofinality $\omega$ but not of the form
$\alpha+\omega$, the conclusion follows from (3) of fact \ref{thm:Buk-Deho}.
\end{proof}

\section{Conclusion and open questions}

We summarize what is now known in terms of equiconsistency:
\begin{enumerate}
\item $\zfc$ is equiconsistent with $V=C^{*}\ne L\,+$ $2^{\aleph_{0}}=\aleph_{1}$.
\item Existence of an inaccessible cardinal is equiconsistent with $V=C^{*}+$
$2^{\aleph_{0}}=\aleph_{2}$.
\item Existence of a measurable cardinal is equiconsistent with $\forall n<\omega(C^{*n}\supsetneq C^{*(n+1)})$
and $C^{*\omega}\vDash\mathrm{ZF+\neg AC}$.
\end{enumerate}
Compared to the results regarding $\mathrm{HOD}$, the following questions
remain open:
\begin{question}
\begin{enumerate}
\item Is it possible, under any large cardinal hypothesis, that $\forall n<\omega$
$C^{*n}\supsetneq C^{*(n+1)}$ and $C^{*\omega}\vDash\zfc$? More
generally, for which ordinals $\alpha$ can we get a decreasing $C^{*}$
sequence of length $\alpha$?
\item Is it possible, under any large cardinal hypothesis, that $\forall n<\omega$
$C^{*n}\supsetneq C^{*(n+1)}$ and $C^{*\omega}\nvDash\mathrm{ZF}$?
\end{enumerate}
\end{question}

A natural first attempt towards answering the first question would
be to try and work in a model with more measurable cardinals. However,
it seems that it would require at least \emph{measurably many} measurables:
in a forthcoming paper \cite{yaar2021short-seq}, we generalize \cite[theroem 5.16]{IMEL}
and our proposition \ref{prop:C* in Prikry ext} and show the following:
\begin{thm}
Assume $V=L\left[\mathcal{U}\right]$ where $\mathcal{U}=\left\langle U_{\gamma}\mid\gamma<\chi\right\rangle $
is a sequence of measures on the increasing measurables $\left\langle \kappa^{\gamma}\mid\gamma<\chi\right\rangle $
where $\chi<\kappa^{0}$. Iterate $V$ according to $\mathcal{U}$
where each measurable is iterated $\omega^{2}$ many times, to obtain
$\left\langle M_{\alpha}^{\gamma}\mid\gamma<\chi,\alpha\leq\omega^{2}\right\rangle $,
with iteration points $\left\langle \kappa_{\alpha}^{\gamma}\mid\gamma<\chi,\alpha\leq\omega^{2}\right\rangle $,
and set $M^{\chi}$ as the directed limit of this iteration. Let $G$
be generic over $V$ for the forcing adding a Prikry sequence to every
$\kappa^{\gamma}$. Set for every $\gamma<\chi$ $E^{\gamma}=\left\langle \kappa_{\omega\cdot n}^{\gamma}\mid1\leq n<\omega\right\rangle $
and $\hat{E}^{\gamma}=\left\langle \kappa_{\omega\cdot n}^{\gamma}\mid0\leq n<\omega\right\rangle $
then 
\begin{align*}
\left(C^{*}\right)^{V} & =M^{\chi}\left[\left\langle E^{\gamma}\mid\gamma<\chi\right\rangle \right]\\
\left(C^{*}\right)^{V\left[G\right]} & =M^{\chi}\left[\left\langle \hat{E}^{\gamma}\mid\gamma<\chi\right\rangle \right].
\end{align*}
\end{thm}

So, if $V=L\left[\mathcal{U}\right]$ as above, $C^{*}$ is of the
form $L\left[\mathcal{U}_{1}\right]\left[G_{1}\right]$ for some sequence
of measures and a sequence of Prikry sequences on it's measures, and
so $C^{*2}$ is again of that form, where we iterated the measures
in $\mathcal{U}$ $\omega^{2}\cdot2$ many times and add Prikry sequences.
So again, as we've done here, we'll get that $C^{*\omega}$ is the
intersection of the models $M_{\omega^{2}\cdot n}^{\chi}$ where we
iterated each measure $\omega^{2}\cdot n$ times . This is due to
the facts that changing the order of iteration between the measures
doesn't change the final result, and that the Prikry sequences ``fall
out'' during the intersection. Now, we don't have a complete analysis
of intersections of iterations by more than one measure, but Dehornoy
proves the following more general fact:
\begin{fact}[{\cite[section 5.3 proposition 3]{dehornoy1978iterated}}]
 For every $\alpha$ let $N_{\alpha}$ be the $\alpha$th iteration
of $V$ by some measure. Assume $\lambda$ is such that for every
$\alpha<\lambda$, $N_{\alpha}\vDash\mathrm{cf}(\lambda)=\omega$
but there is no $\rho$ such that $\lambda=\rho+\omega$. Then if
$M$ is a transitive inner model of $\zfc$ containing $\bigcap_{\alpha<\lambda}N_{\alpha}$,
then there is some $\alpha<\lambda$ such that $N_{\alpha}\con M$.
\end{fact}

So, if we take $N_{\alpha}$ to be the iteration of $V$ by the first
measure in $\mathcal{U}$, we get that $C^{*\omega}$ contains $\bigcap_{\alpha<\omega^{3}}N_{\alpha}$,
but doesn't contain any $N_{\alpha}$ for $\alpha<\omega^{3}$, so
$C^{*\omega}$ cannot satisfy $\mathrm{AC}$. Hence a different approach,
or larger cardinals, would be required to answer this question. 

A different line of inquiry stems from the following fact:
\begin{fact}[\cite{IMEL}]
 If there is a proper class of Woodin cardinals then the theory of
$C^{*}$ is unchanged by forcing.
\end{fact}

So the question whether $C^{*}\vDash V=C^{*}$ cannot be changed under
forcing in the presence of class many Woodin cardinals. If the sequence
of $C^{*\alpha}$ is definable (perhaps up to some ordinal) then this
will also be in the theory of $C^{*}$ (note that on the face of it
even the sequence up to $\omega$ may not be definable). 
\begin{question}
What can be deduced on the sequence of iterated $C^{*}$ from a proper
class of Woodin cardinals?
\end{question}

\bibliographystyle{amsplain}
\bibliography{Bibliography}

\end{document}